\documentclass[12pt,reqno,tbtags]{amsart}
\usepackage{amsmath}
\usepackage{float}
\usepackage{graphicx}
\usepackage{amsthm}
\usepackage{amssymb}
\usepackage{framed}
\usepackage{tikz}
\usepackage[margin=1in]{geometry}
\usetikzlibrary{arrows}

\usepackage{caption,sansmath}
\DeclareCaptionFont{sansmath}{\sansmath}
\newtheorem{thm}{Theorem}[section]
\newtheorem{lem}[thm]{Lemma}

\newtheorem{corollary}[thm]{Corollary}

\newtheorem{conjecture}{Conjecture}[section]

\newtheorem{claim}{Claim}

\newcommand{\brm}[1]{\operatorname{#1}}
\newcommand{\eps}{\varepsilon}
\newcommand{\pl}[1]{\brm{#1}}

\newcommand{\mbf}[1]{\mathbf{#1}}
\newcommand{\bb}[1]{\mathbb{#1}}

\newcommand{\mc}[1]{\mathcal{#1}}

\title[Erd\H{o}s-Ko-Rado theorem for Lagrangians]{A Tur\'an theorem for extensions via an Erd\H{o}s-Ko-Rado theorem for Lagrangians}
\author{Adam Bene Watts}
\address{Physics Department,  MIT}
\email{abenewat@mit.edu}
\author{Sergey Norin}
\address{Department of Mathematics and Statistics, McGill University}
\email{snorin@math.mcgill.ca}
\thanks{Supported by an NSERC grant 418520}
\author{Liana Yepremyan}
\address{Mathematical Institute, University of Oxford} \email{ yepremyan@maths.oxford.ac.uk} 
\thanks{Research supported in part by ERC Consolidator Grant 647678
}
\date{}

\begin{document}
\begin{abstract}
The \emph{extension} of an $r$-uniform hypergraph $G$ is obtained from it by adding for every pair of vertices of $G$, which is not covered by an edge in $G$, an extra edge containing this pair and $(r-2)$ new vertices. In this paper we determine the Tur\'an number of the extension of an $r$-graph consisting of two vertex-disjoint edges, settling a conjecture of Hefetz and Keevash, who previously determined this Tur\'an number for  $r=3$. As the key ingredient of the proof we show that the  Lagrangian of intersecting $r$-graphs is maximized by principally intersecting $r$-graphs for $r \geq 4$. 
\end{abstract}

\maketitle
\section{Introduction}
In this paper we consider $r$-uniform hypergraphs, which we call \emph{$r$-graphs} for brevity. We denote the vertex set of an $r$-graph $ {G}$ by $V( {G})$, the number of its vertices by $\brm{v}( {G})$ and the number of edges by \(\pl{e}( {G})\). (We use $ {G}$ to denote both the $r$-graph itself and its edge set.) An $r$-graph $ {G}$ is called \emph{$F$-free} if it does not contain $F$ as a subgraph. We denote the class of all ${F}$-free $r$-graphs by $\brm{Forb}({F})$.  The \emph{Tur\'{a}n function} $\brm{ex}(n,{F})$ is the maximum size of an ${F}$-free $r$-graph of order $n$:
\[\brm{ex}(n,F) = \max\left\{\pl{e}( {G}) : \brm{v}( {G})=n,\:  {G} \in \brm{Forb}({F}) \right\}.\]
 The \emph{Tur\'{a}n density} of  an $r$-graph $F$ is defined to be the following limit (which was shown to exist by Katona, Nemetz and Simonovits~\cite{katona}):
$$\pi({F})  = \lim_{n \rightarrow \infty}\frac{\brm{ex}(n, {F})}{{n \choose r}}.$$

The \emph{extension} of an $r$-graph $ {F}$ is an \(r\)-graph, denoted by \(\brm{Ext}( {F})\), obtained from \( {F}\) by adding an extra edge for every uncovered pair of vertices containing this pair and $(r-2)$ new vertices. While in general the study of Tur\'an numbers of hypergraphs is a notoriously hard topic, a robust toolkit of stability arguments which can be used to find $\brm{ex}(n,\brm{Ext}( {F}))$, once the maximum Lagrangian of an $F$-free $r$-graph is determined, has been developed in ~\cite{jiang,extensions,gentriangle, pikhurko}. Using such a stability argument the Tur\'an number of the extension of an edgeless $r$-graph has been determined by Pikhurko in~\cite{pikhurko}. Pikhurko's result has been extended in~\cite{jiang, extensions} to determine the Tur\'an number of the extension of all hypergraphs obtained from a fixed $r$-graph by adding sufficiently many isolated vertices. Our result also relies on stability techniques, including  the generic toolkit, which we refer to as \emph{the local stability method}, developed by two of us in~\cite{extensions,gentriangle}. 

In~\cite{hefetzkeevash} Hefetz and Keevash defined the $r$-graph
$K_{r,r}^{(r)}$ to be the extension of the $r$-graph consisting of two disjoint edges. In the same paper the authors determined $\brm{ex}( {K}_{3,3}^{(3)},n)$ for large $n$. To state their result we need to define the balanced blowup \( {T}_5^{3}(n)\) of \( {K}_5^{(3)}\) on \(n\) vertices, where \( {K}_5^{(3)}\) denotes the complete \(3\)-graph on \(5\) vertices. The $3$-graph \( {T}_5^{3}(n)\) is obtained by partitioning the vertex set of size \(n\) into five parts of as equal sizes as possible, and  defining the edges of  \( {T}_5^{3}(n)\) to be the triples of vertices belonging to three distinct parts. 

\begin{thm}[Hefetz and Keevash,~\cite{hefetzkeevash}]\label{hefetzkeevash}For sufficiently large \(n\), \(\brm{ex}(n,  {K}_{3,3}^3) = \brm{e}( {T}_5^3(n))\)  and, moreover, for such $n$ the unique largest \( {K}_{3,3}^3\)-free \(3\)-graph on \(n\) vertices is \( {T}_5^3(n)\).
\end{thm}

In this paper we extend the results of~\cite{hefetzkeevash}  and  determine \(\brm{ex}( {K}_{r,r}^{(r)},n)\) for  all \(r\geq 4\) and large $n$. The structure of extremal hypergraphs is different from case $r=3$.
 We say that a partition $(A,B)$ of the vertex set of an $r$-graph $ {H}$ is a \emph{star partition} if $|e \cap A|=1$ for every $e \in  {H}$. We say that $ {H}$ is a \emph{star} if it admits a star partition.
We denote by $ {S}^{(r)}[A,B]$ the unique maximal $r$-graph which is a star with a partition $(A,B)$. Finally, we denote by $ {S}^{(r)}(n)$ the star on $n$ vertices with the maximum number of edges. (It is easy to see that $\brm{e}(S^{(r)}(n))=(1-1/r)^{r-1}\binom{n}{r}+o(n^r)$, and that if $(A,B)$ is a star partition of $ {S}^{(r)}(n)$ then $|A|=n/r + o(n)$.) We are now ready to state our main result, which confirms the aforementioned conjecture of  Hefetz and Keevash~\cite{hefetzkeevash}.

\begin{thm}\label{mainresult} For every \(r\geq 4\), there exists \(n_0:=n_0(r)\) such that 
	\[\brm{ex}(n,{K}_{r,r}^r) = \brm{e}( {S}^{(r)}(n))\] for all $n > n_0(r)$ and, moreover,  every \( {K}_{r,r}^{(r)}\)-free \(r\)-graph on \(n\) vertices with maximum number of edges is a star.
\end{thm}

The case $r=4$ of Theorem~\ref{mainresult} has been independently established by Wu, Peng and Chen~\cite{WuPengChen}.
The proof of Theorem~\ref{mainresult}, as well as the proof of Theorem~\ref{hefetzkeevash}, uses the stability method and Lagrangians. The \emph{Lagrangian} $\lambda( {F})$ of an \(r\)-graph \( {F}\) is defined as
\[\lambda( {F}) = \max_{p}{\sum_{e\in {F}}{\prod_{v\in e}{p(v)}}},\]
where maximum is taken over all probability distributions on the vertex set \(V( {F})\), that is, the set of functions $p: V( {F}) \to [0,1]$ such that $\sum_{v \in V( {F})}p(v)=1$. 

The Lagrangian function for graphs was introduced by Motzkin and Straus~\cite{motzkinstraus}, who used it to give a new proof of Tur\'{a}n's Theorem. For hypergraphs, it was introduced independently by Frankl and R\"{o}dl \cite{franklrodljumps} and Sidorenko \cite{sidorenko1987}, who also established some important properties of the function. In particular, it was shown by them that for any $r$-graph, the  Lagrangian  is acheived on a subgraph that \emph{covers pairs}, that is, an $r$-graph in which every pair is contained in an edge (for $2$-graphs, this simplifies to maximum sub-clique.) The Langrangian function is  closely related to Tur\'an density of graphs. 
For any two \(r\)-graphs \({F}\)  and \({G}\), any edge-preserving map \(\varphi:V({F})\rightarrow V({G})\) is called \emph{homomorphism}, that is, for every \(f\in {F}\), \(\varphi(F)\in {G}\). An $r$-graph \({G}\) is called \({F}\)-\emph{hom-free} if there is no homomorphism from \({F}\) to \({G}\). The following lemma  was established by Frankl, F\"uredi in~\cite{franklrodljumps} and independently by Sidorenko in ~\cite{sidorenko1987}.
\begin{lem}\label{lem:hom} For any $r$-graph $F$, 
\[\pi({{F}}) = r!\sup_{{G}\in {\brm{Forb}}_{\brm{ hom}}({F})}{\lambda({G})},\]
where ${\brm{Forb}}_{\brm{hom}}({F})$ is the family of all  \(r\)-graphs that  are \({F}\)-hom-free.
\end{lem}

One can further restrict the search of the Tur\'an density of an $r$-graph to the Lagrangian of the family of those $F$-hom-free graphs, which are also \emph{dense}, where we say an $r$-graph $H$ is dense, if for any proper subgraph $H'$, $\lambda(H')<\lambda(H)$. We say that an $r$-graph $ {H}$ is \emph{intersecting} if $e \cap f \neq  \emptyset$ for all $E,F \in  {H}$. The connection between Tur\'an density of  $ {K}_{r,r}^{(r)}$ and the Lagrangians of intersecting $r$-graphs is established in the following lemma, a version of  which for \(r=3\) is present in~\cite{hefetzkeevash}, Theorem 4.1. 
\begin{lem}\label{lem:turandenskrr} For all \(r\geq 3\), \(\pi\left(K_{r,r}^{(r)}\right) = r!\sup_{H \in \mc{H}}\lambda(H)\), where \(\mathcal{H}\) is the family of all intersecting \(r\)-graphs.
\end{lem} 
\begin{proof}
By Lemma~\ref{lem:hom}, \(\pi\left(K_{r,r}^{(r)}\right) = r!\sup_{\mc{H}}\lambda({H})\), over all dense $K_{r,r}^{(r)}$-hom-free \(r\)-graphs \({H}\). It is not hard to see that every intersecting \(r\)-graph is \(K_{r,r}^{(r)}\)-hom-free. For the other direction, suppose \({H}\) is a dense $K_{r,r}^{(r)}$-hom-free \(r\)-graph. Suppose there are two disjoint edges \(f_1\) and \(f_2\) in \({H}\). Since $H$ is dense, for every pair of vertices $v_1\in f_1$ and $v_2\in f_2$, there exists an edge of $H$ covering them, thus creating a homomorphic copy of $K_{r,r}^{(r)}$, a contradiction.
\end{proof}

Thus to determine $ex(n,K_{r,r}^r)$ asymptotically, as we do it in Theorem~\ref{mainresult}, one is required to find the supremum of Lagrangians of intersecting $r$-graphs. And, indeed, a key ingredient of the proof of Theorem~\ref{hefetzkeevash}, Hefetz and Keevash show that the maximum Lagrangian of intersecting $3$-graphs is uniquely achieved by  \( {K}_5^{(3)}\). However, as noted in~\cite{hefetzkeevash}, for  $r \geq 4$ the analogous result does not hold: The maximum Lagrangian of an intersecting \(r\)-graph is not obtained by the complete $r$-graph \( {K}_{2r-1}^{(r)}\) on $2r-1$ vertices. Let $S_1^{(r)}(n)$ denote the intersecting $r$-graph on $n$ vertices consisting of all edges containing some fixed vertex $v$. A direct calculation shows that  $\lambda( {K}_{2r-1}^{(r)}) =\frac{1}{r^r}\binom{2r-1}{r}$, while $$\lim_{n \to \infty} \lambda(S_1^{(r)}(n)) = \frac{1}{r!}(1-r)^{r-1},$$ and the second expression is larger for $r \geq 4$. We show that for $r \geq 4$ the $r$-graphs $S_1^{(r)}(n)$ asymptotically achieve the supremum of Lagrangians of intersecting $r$-graphs. In fact, in the proof of Theorem~\ref{mainresult} we need a slightly stronger result. We say that an $r$-graph $ {H}$ is \emph{principal} if there exists $v \in V( {H})$ such that $v \in e$ for every $e \in  {H}$.

\begin{thm}\label{maxLagrangian} For every \(r\geq 4\) there exists a constant \(c_{r}\) such that  if \( {H}\) is an intersecting, but not  principal \(r\)-graph,  then \(\lambda( {H}) < \frac{1}{r!}\left(\left(1-\frac{1}{r}\right)^{r-1} - c_{r}\right).\)
\end{thm}

Theorem~\ref{maxLagrangian} can be considered as a weighted version of the classical Erd\H{o}s-Ko-Rado theorem which states that for $n \geq 2r+1$ the intersecting $r$-graph on $n$ vertices with the maximum number of edges is principal. Theorem~\ref{maxLagrangian} implies that for sufficiently large $n$ the maximum measure of an intersecting $r$-graph under a non-uniform product measure is achieved by a principal $r$-graph. Let us note that Friedgut~\cite{friedgut} proved a weighted result for $t$-intersecting\footnote{a set system $G$ is \emph{$t$-intersecting} if $|e \cap f| \geq t$ for all $e,f \in G$} set systems using Fourier analytic methods, but, as we consider set systems consisting only of the sets of size $r$, there appears to be no direct  way to derive Theorem~\ref{maxLagrangian} from the results of~\cite{erdos} and vice versa.
The proof of Theorem~\ref{maxLagrangian} relies primarily on the compression techniques, in particular on  the tools  developed by Ahlswede and Khachatrian in their proof of the Complete Intersection Theorem~\cite{complete}. 

We prove Theorem~\ref{maxLagrangian} in Section~\ref{Lagrangian}.
In Section~\ref{stability} we derive Theorem~\ref{mainresult} from Theorem~\ref{maxLagrangian}. As mentioned earlier, our proof relies on the stability method. In Section~\ref{sec:prelim} we use  the tools developed in~\cite{extensions,gentriangle} to reduce the class of $r$-graphs which need to be considered in the proof of Theorem~\ref{mainresult} to \( {K}_{r,r}^{(r)}\)-free $r$-graphs that are close to $S^{(r)}(n)$ in the edit distance and are nearly regular. In Section~\ref{vertexlocalstability} we prove the upper bound on the number of edges for these $r$-graphs.

\subsection{Notation} 

Our notation is fairly standard. Let
 $[n]=\{1,2,\ldots,n\}$. Let $2^{X}$ denote the set of all subsets of set $X$, and let $X^{(k)}$ denote the set of all $k$-element subsets. 
 For an $r$-graph $ {F}$ and $v \in V( {F})$,  the \emph{link of the vertex \(v\)} is defined as $$ L_{ {F}}(v):=\{I \in (V( {F}))^{(r-1)} \: | \: I \cup \{v\} \in  {F} \}.$$  More generally, for  \(I\subseteq V( {F})\) the \emph{link $L_{ {F}}(I)$  of \(I\)} is defined as 
 $$L_{ {F}}(I) := \{J\subseteq V( {F}) \: | \:  J \cap I = \emptyset, I \cup J\in  {F}\}.$$
 We skip the index $ {F}$, whenever $ {F}$ is understood from the context. 
 
For an \(r\)-graph \( {F}\) and a subset \(A\subseteq V( {F})\) we define \( {F}[A]\) to be the \(r\)-graph induced by \(A\), that is, an \(r\)-graph on the vertex set \(A\) and all the edges of \( {F}\) which contain only the vertices of \(A\). 

Given a family $\mathcal{F}$ of $r$-graphs define $$\lambda(\mathcal{F})=\sup_{ {F} \in \mathcal{F}}\lambda( {F}).$$ 

In the next section we will not only consider $r$-graphs, but more general set systems.  Extending the hypergraph notation we say that a set system $G$ is \emph{intersecting} if $e \cap f \neq \emptyset$ for all $e,f \in G$. We say that a set system $G$ is an \emph{($\leq r$)-graph} if $|e| \leq r$ for every $e \in G$. 

\section{Maximum Lagrangian of Intersecting \(r\)-graphs}
\label{Lagrangian}

In this section we prove Theorem~\ref{maxLagrangian}.

For a positive integer $s$, let $[s]^{+}=[s] \cup \{\infty\}$. 
Our central object of study will be a \emph{weighted intersecting set system} or \emph{w.i.s.s.} for short, which is a triple  $(G,s,p)$, where \begin{itemize}
	\item $s$ is a positive integer,
	\item $G \subseteq 2^{[s]}$ is an intersecting $(\leq r)$-graph,
	\item $p: [s]^{+} \to [0,1]$ is a probability distribution\footnote{That is $p(\infty)=1-\sum_{i=1}^{s}p(i)$.}, which is non-increasing on $[s]$.
\end{itemize} 
It will be convenient for us to write $p_{\infty}$ instead of $p(\infty)$. 
For $e \subseteq 2^{[s]}$ such that $|e| \leq r$  and a probability distribution  $p: [s]^{+} \to [0,1]$, define the \emph{weight $w_p(e)$ of $e$} by  $$w_p(e)=\frac{r!}{(r-|e|)!}p_{\infty}^{r-|e|}\prod_{i \in e}p(i).$$
We frequently use probabilistic intuition to estimate $w_p(e)$. Let $\mbf{S}^r_p$ be a multiset of $r$ elements drawn from $[s]^{+}$ independently at random according to the probability distribution $p$. Then $w_p(e)$ is the probability that the restriction of $\mbf{S}^r_p$ to $[s]$ is equal to $e$, and, in particular, has no repeated elements. For an ($\leq r$)-graph $G \subseteq 2^{[s]}$, we define $w_p(G)=\sum_{e \in G}w_p(e)$. 
Thus $w_p(G)$ is the probability that the restriction of $\mbf{S}^r_p$ to $[s]$ is equal to an edge of $G$.

Let us further motivate the technical definition of a weighted intersecting set system above. Let $G$ be an intersecting $r$-graph, and let $S \subseteq V(G)$ be such that $e \cap f  \cap S \neq \emptyset$ for all $e,f \in G$.
Let $G \downarrow S = \{e \cap S | e \in G\}$. Then $G \downarrow S$ is an intersecting $(\leq r)$-graph. Moreover, if $G$ is a maximal intersecting  $r$-graph on $V(G)$ then $G = \{e \in V(G)^{(r)} | e \cap S \in G \downarrow S \}$.
Thus $G \downarrow S$ contains all the essential information about $G$.
 The  $(\leq r)$-graph $G \downarrow S$ is referred to as a \emph{generating set} of $G$ in~\cite{complete}, and in fact our definition of  a weighted intersecting set system is motivated by the definition of generating sets in~\cite{complete}.

 Consider now a probability distribution $p$ on $V(G)$  such that $\lambda(G)= \sum_{e\in G}{\prod_{v\in G}{p(v)}}$. Note that $r!\lambda(G)$ is the probability that a multiset of $r$ elements drawn from $V(G)$ independently at random according to the probability distribution $p$ produces an edge of $G$.  We assume without loss of generality that $S = [s]$ for some positive integer, and that $p$  is non-increasing on $[s]$.\footnote{The condition that $p$ is non-increasing on $[s]$ might appear artificial at the moment, but is a natural requirement in the compression arguments.} Define a probabilistic distribution  $p': [s]^{+} \to [0,1]$ by setting $p'(i)=p(i)$ for $i \in S$, and $p(\infty)=1-\sum_{i=1}^{s}p(i)$. Then $(G \downarrow [s] ,s,p')$ is a w.i.s.s. and $w_{p'}(G \downarrow [s] )$ can be interpreted as the probability that a restriction of the multiset of $r$ elements drawn from $V(G)$ independently at random according to the probability distribution $p$ to $[s]$ produces an edge of $G \downarrow S$.  It follows that $w_{p'}(G \downarrow S) \geq r!\lambda(G)$. If $G \downarrow [s]$ contains an edge of size $\leq r-2$ then the equality is necessarily strict, but  one can construct a sequence of intersecting $r$-graphs $\{G_i\}_{i \in \bb{N}}$ such that $G_i\downarrow [s] = G \downarrow [s]$ and $\lim_{i \to \infty} r!\lambda(G_i)=w_{p'}(G \downarrow [s])$, by increasing the number of vertices of $G$ and reducing the probability of individual vertices in $V(G) - [s]$. As we want to upper bound the maximum Lagrangians of intersecting $r$-graphs, such sequences of increasingly large $r$-graphs with increasing Lagrangians could present a major technical difficulty, as one can not naturally choose an ``optimal" object in them.  Fortunately restricting our attention to weighted intersecting set systems avoids the issue.

Abusing the notation slightly we will say that the $(\leq r)$-graph $G$ is \emph{principal} if $1 \in e$ for all $e \in G$, and \emph{non-principal}, otherwise.\footnote{Note that this definition is slightly different from the definition of principal $r$-graphs given in the introduction.}  We say that a w.i.s.s. $(G,s,p)$ is \emph{(non)-principal} if $G$ is (non)-principal. Let $L_r =\left(1-\frac{1}{r}\right)^{r-1}$. Note that $L_r \geq 1/e$ for every integer $r \geq 2$. For  a principal w.i.s.s. $(G,s,p)$ we have \begin{equation}
\label{e:principal}
w_p(G) \leq rp(1)(1-p(1))^{r-1} \leq L_r. 
\end{equation}

The following is the main technical result of this section, which directly implies Theorem~\ref{maxLagrangian}.

\begin{thm}\label{t:wiss}
	For every integer $r \geq 4$ there exists $c_r>0$ satisfying the following. Let $(G,s,p)$ be a non-principal w.i.s.s. Then $$w_p(G) \leq L_r - c_{r}.$$ 	
\end{thm}

First let us derive Theorem~\ref{maxLagrangian} from Theorem~\ref{t:wiss}.

\begin{proof}[Proof of Theorem~\ref{maxLagrangian} assuming  Theorem~\ref{t:wiss}.]
	Let $G$ be an intersecting $r$-graph, which is not principal. We may assume	$G \subseteq 2^{[s]}$ for some positive integer $s$. By definition of the Lagrangian there exists a probability distribution $p: [s]^{+} \to [0,1]$ with $p(\infty)=0$ such that $w_p(G)=r!\lambda(G).$ By permuting vertices of $G$ one may further assume that $p$  is non-increasing on $[s]$, implying that  $(G,s,p)$ is a w.i.s.s. By Theorem~\ref{t:wiss} we have
	$$\lambda(G) \leq \frac{1}{r!} w_p(G) \leq \frac{1}{r!}\left(  \left(1-\frac{1}{r}\right)^{r-1} - c_{r}\right),$$
	as desired.
\end{proof}

It remains to proof Theorem~\ref{t:wiss}. The proof occupies the remainder of this section.

We say that a non-principal w.i.s.s.  $(G,s,p)$ is a \emph{target} if $w_p(G) \geq w_{p'}(G')$ for every  non-principal w.i.s.s. $(G',s',p')$ such that $s' \leq s$. Moreover, if the equality holds then $s=s'$, $p(s)>0$ and $\sum_{e \in G}\sum_{i \in e} i \leq \sum_{e' \in G'}\sum_{i \in e'} i$. Clearly, it suffices to prove Theorem~\ref{t:wiss} for targets. 

We use the compression technique to show that targets are very structured. Given a $(\leq r)$-graph $G \subseteq 2^{[s]}$ and $1 \leq i < j \leq s$ we define a \emph{compression map} $R_{ij}: G \to  2^{[s]}$, by setting $$R_{ij}(e)= e \setminus \{j\} \cup \{i\},$$ if $j \in e$, $i \not \in e$ and  $e \setminus \{j\} \cup \{i\} \not \in G$, and  $R_{ij}(e)= e$, otherwise. Then $R_{ij}$ is an injection. It is well known (see~\cite[Proposition 2.1]{franklsurvey}) that if $G$ is intersecting, then so is $R_{ij}(G)$. The next lemma shows that targets are essentially always ``compressed".

\begin{lem}\label{l:compress}
	Let $(G,s,p)$ be a target. Then either \begin{itemize}
		\item either $e \setminus \{j\} \cup \{i\} \in G$ for every $e \in G$ and $1 \leq i < j \leq s$ such that $j \in e$, $i \not \in e$, 
		\item or $2 \in e$ for every $e \in G$.
	\end{itemize}
\end{lem}
\begin{proof} Note that if $R_{ij}(G) = G$, for all   $1 \leq i < j \leq s$ then the first outcome of the lemma holds. Thus we suppose that $R_{ij}(G) \neq  G$ for some $1 \leq i < j \leq s$.
	Let $G'=R_{ij}(G)$. Then $(G',s,p)$ is a w.i.s.s. Moreover, $w_p(G') \geq w_{p}(G)$, as $p(i) \geq p(j)$, implying $w_p(R_{ij}(e)) \geq w_p(e)$ for every $e \in G$. Also,
	$\sum_{e \in G}\sum_{i \in e} i > \sum_{e' \in G'}\sum_{i \in e'} i$. Thus  $G'$ is principal, as $(G,s,p)$ is a target. In particular, this implies that $i=1$, i.e. $R_{i'j'}(G) = G$ for all $1 < i' < j' \leq s$.	 Note further that $e \cap \{1,j\} \neq \emptyset$ for every $e \in G$ as $R_{1j}(G)$ is principal.
	
	As $G$ is non-principal there exists $e \in G$ such that $1 \not \in e$, and it follows from the above that $e'=\{2,3,\ldots,|e|+1\} \in G$. This in turn implies that $e''=\{2,3,\ldots,\min(r+1,s)\} \in G$, as $e' \subseteq e''$, and if $e'' \not \in G$ then it could be added to $G$, violating the condition that $G$ is a target.
	We further must have $j \in e''$ and $(e''\setminus \{j\}) \cup \{1\} \not \in G$, as otherwise $e'' \in R_{1j}(G)$ violating the assumption that $R_{1j}(G)$ is principal. 
	
	Suppose now for a contradiction that $2 \not \in f$ for some $f \in G$.  If $1 \not \in f$, then $j \in f$ and $j \neq 2$. Let $f' = f \setminus  \{j\} \cup \{2\}$. Then $f' \in G$, as $R_{2j}(G)=G$, however $f' \cap\{1,j\} = \emptyset$, a contradiction. Thus $1 \in f$. It follows, as above, that $f''=\{1,3,\ldots,\min(r+1,s)\} \in G$. In this case however, we have $R_{1j}(e'')=e''$, contradicting $R_{1j}(G)$ being principal.
\end{proof}

Given a probability distribution $p$ on $[s]^{+}$, define a probability distribution $\bar{p}$ on $[s-1]^{+}$ by setting $\bar{p}(i)=p(i)$ for every $i \in [s-1]$ and $\bar{p}(\infty)={p}(\infty)+p(s)$. For a set system $H$ define $H-s=\{e\backslash\{s\}|e \in H \}$. 

\begin{lem} \label{l:prelimbound}
	Let $(G,s,p)$ be a target for some $s>2$. Then there exist $e, f \in G$ such that $e \cap f = \{s\}$, and we have $e \cup f = [s]$ for each such pair $e,f$. In particular, $s \leq 2r-1$.
\end{lem}

\begin{proof}
	Suppose for a contradiction that $(e \cap f) \setminus \{s\} \neq \emptyset$ for all $e,f \in G$.
	Let $G'=G-s$. Then $G'$ is intersecting, and $w_{\bar{p}}(e\setminus \{s\} ) \geq w_{p}(e\setminus \{s\} )+ w_{p}(e)$ for every $e \in G$, implying  $w_{\bar{p}}(G') \geq w_p(G)$. Thus w.i.s.s. $(G',\bar{p},s-1)$ contradicts the assumption that $(G,s,p)$ is a target. 
	
	Consider now  $e, f \in G$ such that $e \cap f = \{s\}$.
	Suppose for a contradiction that there exists $i \in [s]$ such that  $ i \not \in e \cup f$.  Then $e'=(e \setminus \{s\}) \cup \{i\} \in G$ by Lemma~\ref{l:compress}, as $2 \not \in e \cap f$.
	However, $e' \cap f =\emptyset$, yielding the desired contradiction.
\end{proof}

\begin{corollary}\label{c:compress}
	Let $(G,s,p)$ be a target for some $s>2$. Then $e \setminus \{j\} \cup \{i\}$ for every $e \in G$ and $1 \leq i < j \leq s$ such that $j \in e$, $i \not \in e$. 	
\end{corollary}	
\begin{proof}
	By Lemma~\ref{l:compress} either the corollary holds, or $2 \in e$ for every $e \in G$. However, if $2 \in e$ for every $e \in G$, then $2 \in e \cap f$ for all $e, f \in G$, contradicting Lemma~\ref{l:prelimbound}.
\end{proof}

The main step in the proof of Theorem \ref{t:wiss} involves removing $s$ from every element of $(G,s,p)$, reducing $p(s)$ to $0$, and modifying $G$ so that the resulting set system  is still intersecting, as we did in Lemma~\ref{l:prelimbound} above. We start by analyzing the change in weight of edges after such a modification.

\begin{lem} \label{l:contribution}
	Let $s$ be a positive integer, let $p$ be  a probability distribution on $[s]^{+}$ such that $p(\infty) \geq p(s)> 0$, and let $e \subseteq [s]$, $|e| \leq r$ be such that  $s \in e$. Then	 
	\begin{equation}\label{e:cc2}
	w_{\bar{p}}(e \backslash \{s\}) \geq 2 w_p(e).
	\end{equation}
\end{lem}

\begin{proof}
	Let $l=|e|$.
	Recall that $w_p(e)$ is the probability that the restriction of $\mbf{S}^r_p$ to $[s]$ is $e$.Similarly, $w_{\bar{p}}(e \backslash \{s\})$ is the probability that the restriction of $\mbf{S}^r_p$ to $[s-1]$ is $e \backslash \{s\}$. Let $A$ be the event that the restriction of $\mbf{S}^r_p$ to $[s-1]$ is $e \backslash \{s\}$ and $s$ occurs in $\mbf{S}^r_p$  zero times or twice. 
	Clearly, $\brm{Pr}[A] \leq w_{\bar{p}}(e \backslash \{s\}) - w_p(e)$. We will show that $\brm{Pr}[A] \geq w_p(e)$, implying the lemma. Let $x=\prod_{i \in e}p(i)$. Then 
	\begin{equation}\label{e:contr1}
	w_p(e)=\frac{r!}{(r-l)!}xp_{\infty}^{r-l}
	\end{equation}
	and
	\begin{align}
	\brm{Pr}[A] =& \frac{r!}{(r-l+1)!}p_{\infty}^{r-l+1}\prod_{i \in e \backslash \{s\} }p(i)
	+ \frac{r!}{(r-l-1)!}p_{\infty}^{r-l-1}\frac{p^2(s)\prod_{i \in e \backslash \{s\} }p(i)}{2} \nonumber \\
	=& \frac{r!}{(r-l+1)!}p_{\infty}^{r-l+1} \frac{x}{p(s)} 
	+ \frac{r!}{(r-l-1)!}p_{\infty}^{r-l-1}\frac{xp(s)}{2}
	\label{e:contr2}
	\end{align}
	Combining, (\ref{e:contr1})  and (\ref{e:contr2}), we have
	$$
	\frac{w_{\bar{p}}(e \backslash \{s\}) - w_p(e)}{w_p(e)} \geq \frac{\brm{Pr}[A]}{w_p(e)}  
	= \frac{p_{\infty}}{p(s)(r-l+1)} + \frac{(r-l)p(s)}{2p_{\infty}}.  
	$$
	If $l=r$, then $(w_{\bar{p}}(e \backslash \{s\}) - w_p(e))/w_p(e) \geq 1$, as $p_{\infty} \geq p(s)$.
	Otherwise, $(r-l)\geq (r-l+1)/2$, and $$
	\frac{w_{\bar{p}}(e \backslash \{s\}) - w_p(e)}{w_p(e)} \geq \left(\frac{p_{\infty}}{p(s)(r-l+1)}\right) + \frac{1}{4}\left(\frac{p(s)(r-l+1)}{p_{\infty}}\right) 
	\geq 1, 
	$$
	where the second inequality is AM-GM. The inequality  (\ref{e:cc2}) follows.
\end{proof} 

All the necessary tools in hand, we continue the proof of  Theorem~\ref{t:wiss}. Let $c=1/500$.
We prove by induction on $s$ that if $(G,s,p)$ is a non-principal w.i.s.s. then \begin{equation}\label{e:goal}
w_p(G) \leq L_r - c 2^{-\min(2r-2,s-1)}.
\end{equation} Theorem~\ref{t:wiss} with $c_r = c2^{-2r+2}$ is implied by this statement.

The base case $s=1$ is trivial as every w.i.s.s. $(G,s,p)$ with $s=1$ is principal. 

We divide the proof of the induction step into several cases. In the first  case, Lemma~\ref{l:contribution} and the argument uses the compression techniques and tools developed above. In the remaining cases, the proof is fairly straightforward for large $r$, but the small $r$ cases require brute force computation. 

\vskip 5pt \noindent {\bf Case 1: $s \geq 3$ and  $p(\infty) \geq p(s)$.} 

As $s > 2$, by Lemma~\ref{l:prelimbound}, for every $e \in G$ there exists at most one $f \in G$ such that 
$(e \backslash \{s\}) \cap(f \backslash \{s\}) = \emptyset$. Let $$G_0 = \{e \in G \:|\: f \cap e \neq \{s\} \mathrm{\:  for\: every\:} f \in G \},$$ and let $G'=G -G_0$.  For 

By the preceding observation $G'$ can be partitioned into two $(\leq r)$-graphs $H_1$ and $H_2$ such that if $e,f  \in G$ are such that $(e \backslash \{s\}) \cap(f \backslash \{s\}) = \emptyset$ then $e \in H_i$, $f \in H_{3-i}$ for some $i \in [2]$.
It follows that $(\leq r)$-graphs
$G_1=(G_0-s) \cup (H_1 - s)$ and $G_2=(G_0-s) \cup (H_2 - s)$ are both intersecting. Note further that $(G_0-s) \cap (H_i-s) =\emptyset$ for $i=1,2$. Indeed, for every $e \in H_i$ there exists $f \in H_{3-i}$ such that $(e \setminus \{s\}) \cap f = \emptyset$ implying $e \setminus \{s\} \not \in G$. 
By  applying Lemma~\ref{l:contribution} to every element of $H_i$ we have 
\begin{equation}\label{e:result1}
w_{\bar{p}}(H_i -s) \geq 2 w_p(H_i)
\end{equation}
for $i=1,2$.
Note further, that as in Lemma~\ref{l:prelimbound} we have
\begin{equation}\label{e:result2}
w_{\bar{p}}(G_0-s)  \geq w_p(G_0),
\end{equation}  
Summing (\ref{e:result1}) and (\ref{e:result2}) we obtain
\begin{equation}\label{e:result3}
\frac{w_{\bar{p}}(G_1)+w_{\bar{p}}(G_2)}{2}  \geq \frac{(w_p(G_0)+2w_p(H_1))+(w_p(G_0)+2w_p(H_2))}{2} = w_p(G).
\end{equation} 
Note that at least one of the w.i.s.s. $(G_1,s-1,\bar{p})$ and $(G_2,s-1,\bar{p})$ is non-principal, as $ G_1 \cup G_2 = G-e$.  By Lemma~\ref{l:prelimbound} $s \leq 2r-1$ , and
we suppose by symmetry, that $(G_1,s-1,\bar{p})$ is non-principal. Then $w_{\bar{p}}(G_1) \leq L_r - c2^{-s+2}$ by the induction hypothesis, and  $w_{\bar{p}}(G_2) \leq L_r$, using  the induction hypothesis if $G_2$ is also non-principal.  The inequality (\ref{e:result3}) now implies $$w_{p}(G) \leq \frac{(L_r - c2^{-s+2})+L_r}{2}=L_r - c2^{-s+1},$$
as desired.

\vskip 5pt \noindent {\bf Case 2:  $s=2$.}

As $(G,2,p)$ is non-principal, we have $\{2\} \in G$. Let $x = p(1), y=p(2)$. 
By the AM-GM inequality we have
\begin{align}\label{e:s21}
xy(1-x-y)^{r-2} = \frac{1}{(r-2)^2}((r-2)x)((r-2)y)(1-x-y)^{r-2} \leq \frac{1}{(r-2)^2} \left(\frac{r-2}{r} \right)^r,
\end{align}
and
\begin{align}\label{e:s22}
y(1-2y)^{r-1} = \frac{1}{2(r-1)}(2(r-1)y)(1-2y)^{r-1} \leq \frac{1}{2(r-1)} \left(\frac{r-1}{r} \right)^r.
\end{align} 
Using (\ref{e:s21}) and (\ref{e:s22}) we obtain
\begin{align*}
w_p(G) &\leq r(r-1)xy(1-x-y)^{r-2} + ry(1-x-y)^{r-1} \\&\leq r(r-1)xy(1-x-y)^{r-2} + ry(1-2y)^{r-1} \\& \leq \frac{r(r-1)(r-2)^{r-2}}{r^r}+ \frac{r(r-1)^{r-1}}{2r^r} \\&= L_r\left(\left(\frac{r-2}{r-1}\right)^{r-2}+\frac{1}{2}\right) \\&\leq L_r -\frac{L_r}{18} \leq L_r -c,
\end{align*}
as desired.

\vskip 5pt \noindent {\bf Case 3:  $r \geq 5$ and $p(\infty) \leq p(s)$.} 

Let $z = \sum_{i=1}^{s}p(i)$. Then $z \geq sp(s)$, and so $z  \geq \frac{s}{s+1}$ and therefore $p(\infty) \leq 1/(s+1)$ . We upper bound $w_p(G)$ by the probability of the event $C$ that the restriction of $\mathbf{S}^r_p$ to $[s]$ has no repeated elements. Clearly, given $z \geq \frac{s}{s+1}$, $\brm{Pr}[C]$ is maximized when $p(1)=p(2)=\ldots=p(s)$ as $\brm{Pr}[C]$ is a symmetric multi-linear function of $p(1),\ldots,p(s)$, and it is further maximized when $p(\infty)$ is maximum. Thus we assume that $p(x)=1/(s+1)$ for every $x \in [s]^{+}$. Let $C_0$ be the event  that $\mathbf{S}^r_p$ has no repeated elements at all. Then
\begin{equation}
\brm{Pr}[C_0] = \frac{(s+1)!}{(s+1-r)!(s+1)^{r}}, 
\end{equation}
when $s+1 \geq r$, and $\brm{Pr}[C_0] = 0$, if $s+1 < r$. Let $\bar{C_0}$ denote the negation of the event $C_0$.
Then$\brm{Pr}[C| \bar{C_0}] \leq \frac{1}{s+1}$, by symmetry. If $s+1 < r$ 
we have $$w_p(G) \leq \brm{Pr}[C] = \brm{Pr}[C | \bar{C_0}]  \leq \frac{1}{s+1} \leq \frac{1}{3} \leq L_r -c,$$
as $L_r \geq 1/e$ and $c \leq 1/e-1/3$. 
Thus we assume $s+1 \geq r$. We crudely estimate $w_p(G)$ as follows: 
\begin{align*}
w_p(G) &\leq  \brm{Pr}[C] =   \brm{Pr}[C_0] + \brm{Pr}[C | \bar{C_0}](1- \brm{Pr}[C_0]) \\&\leq \frac{1}{s+1} + \frac{(s+1)!}{(s+1-r)!(s+1)^{r}} \leq \frac{1}{r} + \prod_{i=0}^{r-1}\frac{s+1-i}{s+1} \\ &\leq \frac{1}{r} + \left(\frac{2s+3-r}{2s+2}\right)^{r} \leq \frac{1}{r} + \left(\frac{3r+1}{4r}\right)^{r+1},
\end{align*}
where the third inequality is by AM-GM inequality.
The function $f(r) = \frac{1}{r} + \left(\frac{3r+1}{4r}\right)^{r} $ decreases with $r$, and $f(7)<1/3$. Thus $w_p(G) \leq L_r - (1/e-1/3) \leq L_r -c$ for $r \geq 7$.

The cases $r=5,6$ require more care. We use the precise formula   
$$
\brm{Pr}[C] = \frac{1}{(s+1)^r}\sum_{i=0}^{r}\binom{r}{i}\frac{s!}{(s-i)!}.
$$
and verify that $\brm{Pr}[C] \leq L_r-0.04$ for $r \in \{5,6\}$ and $r \leq s \leq 2r-1$ by computing the corresponding nine values.

\vskip 5pt \noindent {\bf Case 4:  $r=4$ and $s>2$.} 

Suppose first that $\{2,3,4\} \in G$, then by Corollary~\ref{c:compress} every three element subset of $[4]$ is an edge of $G$. Therefore $|e \cap [4]| \geq 2$ for every $e \in G$. 
It follows $w_p(G)$ is upper bounded by the probability of the event $A$ that $\mbf{S}^{r}(p)$ contains at least two elements of $[4]$, but no element of $[4]$ appears twice. As in the previous case, the probability of $A$ is clearly maximized $p(1)=p(2)=p(3)=p(4)=x$ for some $0 \leq x \leq 1/4,$ and so we assume that these equalities hold.
\begin{align*}
w_p(G) &\leq \brm{Pr}[A] \\ &\leq \sum_{i=2}^4\binom{4}{i}\frac{4!}{(4-i)!}x^i(1-4x)^{4-i} \\&=72x^2(1-4x)^2+96x^3(1-4x)+24x^4 \leq 0.41,
\end{align*}
where the last inequality is obtained by explicitly computing the maximum of $72x^2(1-4x)^2+96x^3(1-4x)+24x^4$ on $0 \leq x \leq 1/4.$\footnote{The maximum is equal to $$24 \left(\frac{3 \left(5-\sqrt{3}\right)^4}{21296}-\frac{5 \left(5-\sqrt{3}\right)^3}{2662}+\frac{3}{484} \left(5-\sqrt{3}\right)^2\right) $$ and is achieved at $x = (5 - \sqrt{3})/22$.} As $e_4 = 0.421875$, it follows that $w_p(G) \leq e_4 - 0.01$ in this case.

Thus $\{2,3,4\} \not \in G$. By Corollary~\ref{c:compress}, every edge of $G$ contains at least two elements of $[s]$, including $1$, and no repeated elements of $[s]$, or contains $4$ distinct elements of $\{2,\ldots,s\}$. Let $B$ be the event that  $\mbf{S}^{r}(p)$ produces a multiset with the above properties. Thus $w_p(G) \leq \brm{Pr}[B]$. Once again, $\brm{Pr}[B]$ is maximized when $p(i)=y$ for $i=2,3,\ldots,s$ for some $y$. Let $p(1)=x$. Let $C$ be the event that  $\mbf{S}^{r}(p)$ contains $1$ exactly once. Then \begin{equation}\label{e:last1}\brm{Pr}[C]=4x(1-x)^3, \end{equation} 
\begin{equation}\label{e:last2}\brm{Pr}[B \setminus C]= \frac{(s-1)!}{(s-5)!}y^4,\end{equation} 
and \begin{equation}\label{e:last3}\brm{Pr}[C \setminus B] = 4x(1-x-(s-1)y)^3 + 12(s-1)xy^2(1-x-y)+4(s-1)xy^3.\end{equation} 
As $s \leq 7$, $x \geq y$ and $1-x-y \geq (s-2)y$, we have 
\begin{equation}\label{e:last4}12(s-1)xy^2(1-x-y) \geq  \frac{(s-1)!}{(s-5)!}y^4.\end{equation} Moreover, 
\begin{align}\label{e:last5}4&x(1-x-(s-1)y)^3 + 4(s-1)xy^3 \notag\\ &\geq \frac{4}{(s-1)^2}x ((1-x-(s-1)y)^3 + ((s-1)y)^3) \notag \\& \geq \frac{8}{36}x\left(\frac{1-x}{2}\right)^3
\end{align}
Combining (\ref{e:last1})--(\ref{e:last5}), we obtain
$$w_p(G) \leq \brm{Pr}[B] = \brm{Pr}[C] +\brm{Pr}[B \setminus C] - \brm{Pr}[C \setminus B] \leq  \left(4 - \frac{1}{36} \right)x(1-x)^3 \leq e_4 - e_4/144 \leq e_4 - c_4,$$
as desired.

\section{Stability: Proof of Theorem~\ref{mainresult}}
\label{stability}
\subsection{Local Stability}
\label{sec:prelim}
In this section we introduce the result from~\cite{extensions}, which builds on the techniques originally presented in~\cite{gentriangle}, and allows us to reduce the proof of Theorem~\ref{mainresult} to $r$-graphs which are ``close'' to the conjectured extremum.

We say that an $r$-graph $ {G}$ is obtained from an $r$-graph $ {F}$ by \emph{cloning a vertex $v$ to a set $W$} if $ {F} \subseteq  {G}$, $V( {G}) \setminus V( {F}) = W\setminus\{v\}$, and $ L_{ {G}}(w)= L_{ {F}}(v)$ for every $w \in W$. We say that $ {G}$ is \emph{a blowup of $ {F}$} if $ {G}$ is isomorphic to an $r$-graph obtained from $ {F}$ by repeatedly cloning and deleting vertices. We denote the set of all blowups of $ {F}$ by $\mathcal{B}( {F})$. We say that a family of $r$-graphs is \emph{clonable} if it is closed under the operation of taking blowups. Note that the family of all stars is clonable.

For a family of $r$-graphs $\mathcal{F}$, let 
$$m(\mathcal{F},n):=\max_{\substack{ {F} \in \mathcal{F} \\ \brm{v}({ {F}}) = n}} | {F}|$$
denote the maximum number of edges in an $r$-graph in  $\mathcal{F}$ on $n$ vertices.

Let $\mathcal{F}$ and $\mathcal{H}$ be two families of $r$-graphs. We define \emph{the distance $d_{\mathcal{F}}( {F})$ from an $r$-graph $ {F}$ to a family $\mathcal{F}$} as 
\[ d_{\mathcal{F}}( {F}):=\min_{\substack{ {F'}\in\mathcal{F} \\ \brm{v}( {F}) = \brm{v}( {F}')}}{| {F}\triangle {F}'|}.\]

For $\eps, \alpha>0 $, we say that $\mathcal{F}$ is $(\mathcal{H}, \eps, \alpha)$-\emph{locally stable} if there exists $n_0 \in \mathbb{N}$ such that for all $ {F}\in\mathcal{F}$ with $\brm{v}( {F}) =n \geq n_0$ and $d_{\mathcal{H}}( {F})\leq \eps n^r$ we have
\begin{equation}\label{eq:localstability}
| {F}|\leq m(\mathcal{H},n) - \alpha d_{\mathcal{H}}( {F}).
\end{equation}
We say that $\mathcal{F}$ is $\mathcal{H}$-\emph{locally stable}  if $\mathcal{F}$ is $(\mathcal{H}, \eps, \alpha)$-locally stable for some choice of $\eps$ and $\alpha$. We say that $\mathcal{F}$ is $(\mathcal{H}, \alpha)$-\emph{stable} if it is $(\mathcal{H}, 1, \alpha)$-\emph{locally stable}, that is the inequality (\ref{eq:localstability}) holds for all  $ {F}\in\mathcal{F}$ with $\brm{v}( {F}) =n \geq n_0$. We say that \(\mathcal{F}\) is \emph{\(\mathcal{H}\)-stable}, if \(\mathcal{F}\) is \((\mathcal{H}, \alpha)\)-stable for some choice of \(\alpha\). We refer the reader to~\cite{gentriangle} for the detailed discussion of this notion of stability and its differences from the classical definition. 

For $\eps, \alpha>0$, we say that a family $\mathcal{F}$ of $r$-graphs is  $(\mathcal{H}, \eps, \alpha)$-\emph{vertex locally stable} if there exists $n_0 \in \mathbb{N}$ such that for all $ {F}\in\mathcal{F}$ with $\brm{v}( {F}) =n \geq n_0$,  $d_{\mathcal{H}}( {F})\leq \eps n^{r}$, and 
$| L_{ {F}}(v)| \geq r(1-\eps)m(\mathcal{H},n)/n$  for every $v \in V( {F})$, we have
\[| {F}|\leq m(\mathcal{H},n) - \alpha d_{\mathcal{H}}( {F}).\]
We say that $\mathcal{F}$ is  $\mathcal{H}$-\emph{vertex locally stable} if $\mathcal{F}$ is  $(\mathcal{H}, \eps, \alpha)$-vertex locally stable for some $\eps,\alpha$. 
It is shown in~\cite{gentriangle} that vertex local stability implies local stability under mild conditions.

Let ${M}^{(r)}_2$ denote the $r$-graph consisting of two vertex disjoint edges. Note that \( {K}_{r,r}^{(r)} = \brm{Ext}({M}^{(r)}_2)\)and that $\brm{Forb}( {M}^{(r)}_2)$ is exactly the class of intersecting $r$-graphs.

We are now ready to state the main result from~\cite{extensions} that we will be using this paper. The result in full generality requires one to extend the notions of distance and stability to weighted $r$-graphs. In the interest of brevity, we do not present the corresponding definitions and instead state a direct corollary of~\cite[Corollary 2.8]{extensions}, which is necessary for our purposes. An interested reader can verify that the next theorem is indeed a direct weakening of~\cite[Corollary 2.8]{extensions}.   

\begin{thm}\label{thm:extensions} Let $ {G}$ be an $r$-graph, let $\mathcal{F}=\brm{Forb}(\brm{Ext}( {G}))$, let $\mathcal{F}^*= \brm{Forb}( {G})$, and let $\mathcal{H} \subseteq \mathcal{F}$ be a clonable family of $r$-graphs. If the following conditions hold
	\begin{description}
		\item[(T1)] $\mathcal{F}$ is  $\mathcal{H}$-vertex locally stable,
		\item[(T2)] there exists a constant $c>0$ such that $\lambda( {F}) \leq  \lambda(\mathcal{H})-c$ for every $ {F} \in \mathcal{F^*} - \mathcal{H}$.
	\end{description}
then $\mathcal{F}$ is $\mathcal{H}$-stable. In particular, there exists $n_0 \in \bb{N}$ such that if $ {F} \in \mathcal{F}$ satisfies $\pl{v}( {F})=n$ and  $| {F}| =m(\mathcal{F},n)$ for some $n \geq n_0$ then $ {F} \in \mathcal{H}$.	
\end{thm}

Let $\mathcal{S}$ denote the family of $r$-graphs which are  stars.
To derive Theorem~\ref{mainresult} from Theorem~\ref{thm:extensions} it suffices to show that (T1) and (T2) hold when $\mathcal{F}=\brm{Forb}( {K}_{r,r}^{(r)})$, $\mathcal{F}^*$ is  the family of  all intersecting $r$-graphs, and \(\mathcal{H}=\mathcal{S}\). The validity of condition (T2) in this case follows directly from Theorem~\ref{maxLagrangian}. Thus it remains to verify condition (T1). It will follow from the theorem we introduce next. 

Let \[d_r :=\frac{\left(1-\frac{1}{r}\right)^{r-1}}{(r-1)!},\]
\[e_r:=\frac{\left(1-\frac{1}{r}\right)^{r-1}}{r!}.\] 
Then $m(\mathcal{S},n)=e_r n^r+o(n^r)$. 

\begin{thm}\label{vertexstab} For every $r \geq 2$ there exist $n_0  \in \bb{N}$  and $\delta >0$ such that the following holds. Let $ {F}$ be a $ {K}_{r,r}^{(r)}$-free $r$-graph with $v( {F})=n \geq n_0$. If
	\begin{description}
		\item[(S1)] $|L_{ {F}}(v)| \geq (1-\delta)d_r n^{r-1}$ for every $v \in V( {F})$, and
		\item[(S2)] there exists a star $ {S}$ such that $| {F} \triangle  {S}| \leq \delta n^r$, 
	\end{description}
then $ {F}$ is a star. 	
\end{thm}

It is easy to see that Theorem~\ref{vertexstab} implies that $\brm{Forb}( {K}_{r,r}^{(r)})$ is $\mathcal{F}$-vertex locally stable.\footnote{In fact, Theorem~\ref{vertexstab} is stronger as it implies that $\brm{Forb}( {K}_{r,r}^{(r)})$ is $(\mathcal{F}, \delta,\alpha)$-vertex locally stable for \emph{every} $\alpha >0$.} Thus it remains to prove Theorem~\ref{vertexstab}. The proof occupies the rest of the section.

\subsection{Vertex Local Stability: Proof of Theorem~\ref{vertexstab}} 
\label{vertexlocalstability}

Before we start the main proof, let us state and prove two auxiliary lemmas which will need later. Our first lemma ensures that if a large star $S$ has edge  density close to the maximum possible (i.e. $e_rn^r$), then the star-partition of $S$ is
almost as balanced as the one for $S^{(r)}(n)$. More precisely, given a star $S$, with star-partition $(A,B)$ we call $S$, \emph{$\eps$-balanced}, if $|A-\frac{n}{r}|\leq \eps$.

\begin{lem}\label{lem:balanced}Let $r\geq 3$.  For every $\eps>0$ there exists $\delta>0$ and $n_0$  such that if $S$ is a star  on $n\geq n_0$ vertices such that $e(S)\geq (e_r - \delta)n^r$ then  $S$ is $\eps$-balanced.
\end{lem}
\begin{proof} Let $(A, B)$ be the star-partition of $S$, then  $|S|\leq |A|(n-|A|)^{r-1}/(r-1)!$. Setting $|A|/n=x$ and $f(x)=x(1-x)^{r-1}/(r-1)!$ we rewrite the above inequality as $|S|\leq f(x)n^r$. Note that $f:[0,1] \to [0,1]$ is a continuous function with the unique maximum $e_r$ achieved at $x^*=1/r$. It follows that for every $\eps>0$ there exists $\delta >0$ such that $f(x) \geq e_r - \delta$ implies $|x - x^*| \leq \eps$, implying the desired inequality.
\end{proof}

Let the constants $n_0,k,\eps_1,\eps_2,\delta$  be chosen implicitly to satisfy the inequalities appearing throughout the proof. It will be clear that these inequalities are satisfied as long as  $$ \frac{1}{n_0} \ll \delta \ll \eps_1 \ll \eps_2 \ll \frac{1}{k} \ll \frac{1}{r}.$$ 

%Let the constant $M_{\ref{lem:missingedges}}>0$ be derived from Lemma~\ref{lem:missingedges}. 
Let $ {S}$ be the star satisfying (S2), let $(A,B)$ be the star partition of $ {S}$, and let  $ {S}_{c}$ be the complete star with the partition $A$ and $B$. It follows from (S1) that $| {F}| \geq (1-\delta)e_r n^r$, and thus $| {S}|\geq (e_r-2\delta)n^r$. Therefore, $| {S}_{c}  -  {S}| \leq 2\delta n^r$, as $|{ {S}}_{c}| \leq e_rn^r$. Finally, it follows that $| {F} \triangle  {S}| \leq 3\delta n^r$. By replacing  $\delta$ with $3 \delta$, we may assume that $ {S}={ {S}}_c$. 

We say that a pair $(A',B')$ with $A' \subseteq A, B' \subseteq B$ is  \emph{perfect} if $ {F}[A' \cup B']= {S}^{(r)}[A',B']$, i.e. the restriction of $ {F}$  
to $A' \cup B'$ coincides with the restriction of $ {S}$. Let $\mc{P}$ denote the set of all perfect pairs. 

For a positive integer $k$, let the random variables $\mbf{X}^k$ and $\mbf{Y}^k$ be subsets of size $k$ of $X\subseteq A$ and $Y\subseteq B$, respectively, chosen uniformly and independently at random.
We say that a vertex $v \in V( {F})$ is \emph{$(A,k,\eps)$-regular} if
$$\brm{Pr}[(\mbf{A}^{k-1} \cup \{v\}, \mbf{B}^k) \in  \mc{P} ] \geq 1 -\eps.$$
Similarly, a vertex $v \in V( {F})$ is \emph{$(B,k,\eps)$-regular} if 
 $$\brm{Pr}[(\mbf{A}^k , \mbf{B}^{k-1} \cup \{v\}) \in \mc{P} ] \geq 1 -\eps.$$
 
 The next claim motivates the above definitions. (It will be applied with $\eps=\eps_1$ and $\eps = \eps_2$ later in the proof.)
 
 \begin{claim}\label{cl:regular} 	If $\eps \ll 1/r$ then
	\begin{description}
  		\item[(C11)] if $v_1,v_2$ are distinct $(A,k,\eps)$-regular vertices then there exists no $e \in  {F}$ such that $\{v_1,v_2\} \subseteq e$,
        \item[(C12)] if $v_1,v_2, \ldots, v_r$ are $(B,k,\eps)$-regular then  $\{v_1,v_2,\ldots,v_r\} \not \in  {F}$.
	\end{description}
 \end{claim}
 
 \begin{proof} 
 	
 {\bf (C11):} Suppose for a contradiction that there exists $e \in \mc{F}$ such that $\{v_1,v_2\} \subseteq e$. By $(A,k,\eps)$-regularity of $v_1$ and $v_2$ we have   
 $$\brm{Pr}[(\mbf{A}^{k-1} \cup \{v_1\}, \mbf{B}^k), (\mbf{A}^{k-1} \cup \{v_2\}, \mbf{B}^k) \in \mc{P}] \geq 1 -2\eps.$$
 If $\eps <1/2$ and $k > r^3 + r$  there exist  $(A'  \cup \{v_i\}, B')\in \mc{P}$ for some $A',B'$ with $|A'|\geq (r-1)^2+r$ and $|B'| \geq 2r+(r-2)(r^2+1)$ and $i=1,2$. Let $f_i \subseteq B' \cup \{v_i\}$ for $i=1,2$ be chosen so that $e \cap f_i = \{v_i\}$ and $f_1 \cap f_2 =\emptyset$ .  One can then straightforwardly find an extension of $\{f_1,f_2\}$ in $F[A' \cup B' \cup \{v_1,v_2\}] \cup \{e\}$, by using $e$ to extend the pair $\{v_1,v_2\}$, and selecting the edges to extend other pairs of vertices greedily in  $F[A' \cup B']$. Thus we obtain the desired contradiction as $F$ is $K^{(r)}_{r,r}$-free. 
 
 \vskip 10pt
 {\bf (C12):} The proof is similar to (C11). Suppose that there exists $e=\{v_1,\dots, v_r\}\in F$, such that  $v_1,v_2, \ldots, v_r$ are $(B,k,\eps)$-regular. As in (C11), if $\eps < 1/r$ we can find sufficiently large $A'\subseteq A, B'\subseteq B$ such that  $(A' , B'\cup\{v_i\})\in \mc{P}$ for $i \in [r]$. Choose $f \subseteq A' \cup B'$, such that $|f \cap A'| = 1$ and $f \cap e = \emptyset$. Then one can find an extension of $\{e,f\}$ in $\mc{F}[A' \cup B' \cup \{v_1,\ldots, v_r\}]$, choosing edges to extend pairs of vertices in $e$ and $f$ greedily.  
  \end{proof}
 
 Note that Claim~\ref{cl:regular} implies that, if every vertex of $ {F}$ is either $(A,k,\eps)$-regular or  $(B,k,\eps)$-regular for some $k,\eps$ satisfying the conditions of Claim~\ref{cl:regular}, then $ {F}$ is a star, as desired. Thus our goal is to show that all vertices are ``sufficiently regular". We start by showing that almost all vertices are.
 
 \begin{claim}\label{cl:A0B0} 	There exist $A_0 \subseteq A, B_0 \subseteq B$ such that 
	\begin{description}
 		\item[(C21)]
		$|A_0| \geq (1/r - \eps_1)n, |B_0| \geq ((r-1)/r - \eps_1)n$,
		\item [(C22)] every $v \in A_0$ is  $(A,k,\eps_1)$-regular, and
 \item[(C23)] every $v \in B_0$ is  $(B,k,\eps_1)$-regular. 
	\end{description}
 \end{claim}
 \begin{proof} Let $\bf{e}$ be an edge of $S$ chosen uniformly at random. 
 $$\brm{Pr}[\mbf{e} \not \in F] \leq \frac{|F \triangle S|} {|S|} \leq \frac{\delta}{e_r-2\delta} \leq \frac{2}{e_r}\delta.$$
  
Therefore, \begin{equation}\label{e:c21}
\brm{Pr}[(\mbf{A}^{k},\mbf{B}^{k})\notin \mc{P})] \leq  {k}\binom{{k}}{r-1} \brm{Pr}[\mbf{e} \not \in F] \leq   {k}\binom{{k}}{r-1}\frac{2}{e_r}\delta,
\end{equation}
where the first inequality holds, as one can choose $\mbf{e}$ by first choosing the pair $(\mbf{A}^{k},\mbf{B}^{k})$,  and then choosing an $r$-element $\mbf{e} \subseteq \mbf{A}^{k} \cup \mbf{B}^{k}$ such that $|\mbf{e} \cap  \mbf{A}^{k}| =1$ uniformly at random.
 
Let $A_0$ be the set of all vertices in $A$ which are $(A,k,\eps_1)$-regular. Then
\begin{equation}\label{e:c22}
\brm{Pr}[(\mbf{A}^{k},\mbf{B}^{k}) \notin \mc{P})] \geq \eps_1\frac{|A|-|A_0|}{|A|},
\end{equation}
as $\mbf{A}^{k}$ can be chosen by first choosing a single element of $v \in A$ uniformly at random to be in it, and if such $v$ is not in $A_0$, then the resulting pair $ (\mbf{A}^{k},\mbf{B}^{k})$ is not perfect with probability at least $\eps_1$.

Combining, (\ref{e:c21}) and (\ref{e:c22}) we obtain
$$|A_0|\geq \left(1-{k}\binom{{k}}{r-1}\frac{2}{\eps_1 e_r}\delta\right) |A| \geq \left(1 -\frac{\eps_1}{2}\right)|A|.$$
Moreover, by Lemma~\ref{lem:balanced}, we have $|A| \geq  (1/r - \eps_1/2)n$, as $\delta \ll \eps_1$. It follows that $|A_0| \geq  (1/r - \eps_1)n$, as desired.

Analogously, we have $|B_0| \geq ((r-1)/r - \eps_1)n$, where $B_0$ is the set of all  $(B,k,\eps_1)$-regular vertices in $B$, finishing the proof of the claim.
 \end{proof}
Our final step is to show  that all vertices of $F$ are $(A,k,\eps_2)$-regular or $(A,k,\eps_2)$-regular for some $\eps_1 \ll \eps_2 \ll 1/k$.\begin{claim}\label{cl:regular2} Let $v \in V$.
	\begin{description}
		\item[(C31)] if there exists $f_0 \in L(v)$ such that $f_0 \cap A_0 \neq \emptyset$ then $v$ is  $(B,k,\eps_2)$-regular, 	
		\item[(C32)] otherwise, $v$ is  $(A,k,\eps_2)$-regular.
	\end{description}
 \end{claim}
\begin{proof} We start by proving (C31). Let $e \in F$, $u \in A_0$ be such that $\{v,u\} \in e$.

Let $B_1 = \{b \in B | |L(v,b)| \geq n^{r-5/2} \}.$ Suppose that there exists $f_1 \in L(v)$ such that $f_1 \subseteq B_1$, $f_1 \cap e = \emptyset$ and  $(A',B') \in \mc{P}$ such that $u \in A'$, $|A'|,|B'|\gg r$ and  $f_1 \subseteq B'$. Choosing $f_2 \in B'$ with $|f_2|=r-1, f_2 \cap (f_1 \cup e) = \emptyset$ one can find an extension of $\{\{v\} \cup f_1, \{u\} \cup f_2 \}$ in $F$, using $e$ to extend $\{v,u\}$, extending the pairs of vertices not containing $v$ in $F[A' \cup B']$ greedily, which can be done as $(A',B') \in \mc{P}$, and extending the remaining pairs containing $v$ greedily using the fact $f_2 \subseteq B_1$. Thus no such choice of $f_1$ and $(A',B')$ is possible. 

Let $b$ be an element of $B-B_1$, and let $\mbf{f}_b$ be a random subset of $B$ of size $r-1$ containing $b$.
Then  
\begin{equation}\label{e:c3f1}
\brm{Pr}[\mbf{f}_b \in L(v)] = \frac{|L(v,b)|}{(|B|-1)^{r-2}} \leq \frac{n^{r-5/2}}{(n/2)^{r-2}} = \frac{2^{r-2}}{\sqrt{n}}.
\end{equation}
Let $\mbf{f}$ now be a random subset of $B$ of size $r-1$ then it follows from (\ref{e:c3f1}) that
\begin{equation}\label{e:c3f2}
\brm{Pr}[\mbf{f} \in L(v) | \mbf{f} \not \subseteq B_1]  \leq \frac{(r-1)2^{r-2}}{\sqrt{n}}
\end{equation}
Note next that $\mbf{f}$  can be chosen by choosing a random pair $(\mbf{A}^{k-1},\mbf{B}^{k})$ and randomly choosing $\mbf{f} \subseteq \mbf{B}^{k}$. Using (\ref{e:c3f2}) and the discussion above, we have
\begin{align*}\brm{Pr}[\mbf{f} \in L(v)] &= \brm{Pr}[\mbf{f} \in L(v) \land \mbf{f} \subseteq B_1 \land \mbf{f} \cap e = \emptyset] +\brm{Pr}[\mbf{f} \cap e \neq \emptyset] + \brm{Pr}[\mbf{f} \in L(v) \land \mbf{f} \not \subseteq B_1] \\ &\leq \brm{Pr}[(\mbf{A}^{k-1} \cup \{u\},\mbf{B}^{k}) \not \in \mc{P}] + \brm{Pr}[\mbf{f} \cap e \neq \emptyset] + \brm{Pr}[\mbf{f} \in L(v) | \mbf{f} \not \subseteq B_1] \\ &\leq \eps_1 + \frac{(r-1)r}{|B|} + \frac{(r-1)2^{r-2}}{\sqrt{n}} \leq 2\eps_1. \end{align*}

Thus $|L(v) \cap B^{(r-1)}| \leq 2\eps_1 n^{r-1}$. By (C11) no element of $L(v)$ contains two elements of $A_0$.

Let $T =  \{f \in V(F)^{(r-1)} | |f \cap A|=1\}$, and let $L_0(v) =  L(v) \cap T$. By the above we have, that every element of $L(v) - L_0(v)$ is contained in $B^{(r-1)}$ or contains an element of $A - A_0$. Therefore $$|L(v) - L_0(v)| \leq  |A-A_0|n^{r-2} +  2\eps_1 n^{r-1} \leq 3\eps_1 n^{r-1}.$$
Thus $|L_0(v)| \geq 
 (d_r-\delta-3\eps_1 ) n^{r-1}$.
Now let $\mbf{f} \subseteq T$ be chosen uniformly at random. Note that $$\brm{Pr}[\mbf{f}  \not \in L_0(v)] \leq 1 - \frac{|L_0(v)|}{|T|} \leq  1 -\frac{d_r-\delta-3\eps_1}{d_r+r\eps_1} \leq \frac{(r+4)\eps_1}{d_r}.$$

Thus
   \begin{align}\label{e:c31}
\brm{Pr}&[(\mbf{A}^{k},\mbf{B}^{k-1}\cup\{v\})\notin \mc{P})] \notag\\&\leq  \brm{Pr}[(\mbf{A}^{k},\mbf{B}^{k-1})\not \in \mc{P})] +  \brm{Pr}[ \exists f \subseteq \mbf{A}^{k} \cup \mbf{B}^{k-1} : f \in T - L_0(v)] \notag \\& {k}\binom{k-1}{r-1}\frac{2}{e_r}\delta + {k}\binom{k-2}{r-2}\frac{(r+4)\eps_1}{d_r} \notag\\&\leq \eps_2,
\end{align}
  as desired.
  
It remains to prove (C32). However,  if there exists no  $I_0 \in L(v)$ such that $I_0 \cap A_0 \neq \emptyset$, then paralleling the computations above we get $|L(v) \cap B^{(r-1)}| \geq  (d_r-\delta-\eps_1 ) n^{r-1}$, and one concludes that $v$ is $(A,k,\eps_2)$-regular using a computation analogous to (\ref{e:c31}).  
\end{proof}

As mentioned earlier, the proof of Claim 3 concludes the proof of Theorem~\ref{vertexstab}. Indeed, Claim 3 implies that $V(F)$ can be partitioned into the set $A'$ of $(A,k,\eps_2)$-regular vertices and a set $B'$ of $(B,k,\eps_2)$-regular vertices. It follows from Claim 1 that $(A',B')$ is a star partition of $F$, and so $F$ is a star.  

\section{Concluding Remarks}

\subsection*{Towards a Complete Intersection Theorem for Lagrangians.}

As mentioned in the introduction our proof of Theorem~\ref{maxLagrangian} relies significantly on the techniques introduced by Ahlswede and Khachatrian in their proof of the Complete Intersection Theorem~\cite{complete}, which determines the maximum number of edges in a $t$-intersecting $r$-graph on $n$ vertices. It seems therefore natural to ask whether Theorem~\ref{maxLagrangian} can be extended to determine the supremum of Lagrangians  of $t$-intersecting $r$-graph. To continue the discussion we need to recall the statement of the Complete Intersection Theorem. For integers $t \geq 1,i \geq 0$, and $r \geq t+i$, let $F(r,t,i)$ denote the $(\leq r)$-graph with $ F(r,t,i) \subseteq 2^{[t+2i]}$ such that $t+i \leq |e| \leq r$ for every $e \in F(r,t,i)$. Then it is easy to see that $F(r,t,i)$ is $t$-intersecting. Consider now $n \geq 2t+i$ and an $r$-graph $G(r,t,i,n)$ consisting of all edges $e \subseteq [n]$, $|e|=r$ such that $e \cap [2t+i] \in F(r,t,i)$. Then $G(r,t,i,n)$ is $t$-intersecting $r$-graph. 
\begin{thm}[The Complete Intersection Theorem~\cite{complete}]\label{t:complete}
 Let $G$ be a $t$-intersecting $r$-graph with $\brm{v}(G)=n$ then 
  $$|G| \leq \max_{0 \leq i \leq r-t}|G(r,t,i,n)|.$$	
\end{thm}
 
The analogous statement for Lagrangians is best stated in terms of weighted $t$-intersecting set systems. Define the weighted $t$-intersecting set system $(G,s,p)$ analogously to our definition of a weighted intersecting set system in Section~\ref{Lagrangian}, except that the $(\leq r)$-graph $G$ is required to be $t$-intersecting. 

\begin{conjecture}\label{c:complete}
Let $G$ be a $t$-intersecting $r$-graph then $$r!\lambda(G) \leq \max_{0 \leq i \leq r-t} w_p(F(r,t,i)),$$ 
where the maximum is implicitly taken over all $p:[t+2i]^{+} \to [0,1]$ such that $(F(r,t,i),t+2i,p)$ is a weighted $t$-intersecting set system. 	
\end{conjecture}

The validity of Conjecture~\ref{c:complete} for $t=1$ follows from Theorem~\ref{maxLagrangian} for $r \geq 4$. For $r=3$, it follows from the results in~\cite{hefetzkeevash} mentioned in the introduction. It is also easy to verify for $r=2$.

Theorem~\ref{t:complete} is relatively easy to establish for $n \gg r,t$. In this regime the maximum $t$-intersecting $r$-graphs are principal, i.e. consist of all set of size $r$ containing a fixed set of size $t$. One would expect the same phenomenon to hold in our setting, i.e. for $r \gg t$, the weighted $t$-intersecting set system of maximum weight consists of single edge of size $t$. The validity of this intuition for $t=1$ is supported by Theorem~\ref{maxLagrangian}. Surprisingly, the first author have shown that the above statement is false for large $t$.

\subsection*{Vertex Local Stability.}

We have taken slightly non-standard route in the proof of Theorem~\ref{vertexstab} by considering a vertex to be well-behaved if probability that a sample containing this vertex matches the expected structure. A different proof of   Theorem~\ref{vertexstab} is given in the third author's PhD thesis~\cite[Theorem 6.2.1]{lianathesis}.

\subsection*{Acknowledgement}

Some of this research was performed in Summer 2015, when the first author was an undergraduate student at McGill University supported by an NSERC USRA scholarship.

\bibliographystyle{amsplain}
\bibliography{lib}

\end{document}